\documentclass[12pt]{amsart}
\usepackage[dvips]{graphicx}
\usepackage{epsfig}
\usepackage{amssymb}
\usepackage[usenames, dvipsnames]{color}
\usepackage{hyperref}
\usepackage{verbatim}
\usepackage[normalem]{ulem}

\theoremstyle{plain}
\newtheorem{theorem}{Theorem}[section]
\newtheorem{lemma}[theorem]{Lemma}

\newtheorem{proposition}[theorem]{Proposition}

\theoremstyle{definition}

\theoremstyle{remark}

\def\La{\Lambda}

\def\th{\theta}

\def\e{\epsilon}

\def\p{\partial}

\def\al{\alpha}

\def\O{\Omega}

\def\r{\rho}

\def\be{\begin{equation}}
\def\ee{\end{equation}}
\def\bes{\begin{equation*}}
\def\ees{\end{equation*}}
\def\bali{\begin{aligned}}
\def\eali{\end{aligned}}
\def\al{\begin{aligned}}
\def\eal{\end{aligned}}

\def\lab{\label}

\def\2O{\underline{\O}}

\numberwithin{equation}{section}

\makeatletter
\def\dashint{\operatorname%
{\,\,\text{\bf--}\kern-.98em\DOTSI\intop\ilimits@\!\!}}
\makeatother

\begin{document}


\title[partial type I]{ On partial type I solutions to the Axially symmetric Navier-Stokes equations
}

\thanks{}

\author[Q. S. Zhang]{Qi S. Zhang}

\address[]{Department of mathematics, University of California, Riverside, CA 92521,
USA}

\email{qizhang@math.ucr.edu}



\begin{abstract}

 Let $v= v_{r}e_{r} + v_{\th}e_{\th} + v_{3}e_{3}$ be a Leray-Hopf solution to the axially symmetric Navier-Stokes equations (ASNS). We call it a partial type I solution if $v_r(x, t) \ge -C/\sqrt{T-t}$ for some constant $C>0$ and  $(x, t) \in \mathbf{R}^3 \times [0, T)$. In this paper, it is proven that such solution does not blow up at time $T$ under the extra mild assumption that $|v_\theta(x, 0)| |x'|$ is bounded. This extends a well known result by two groups of people who proved the no blowup conclusion under the full type I condition:  $|v(x, t)| \le C/\sqrt{T-t}$. The result also confirms the physical intuition that potential blow ups for ASNS are caused by super-critical inward radial velocity.
\end{abstract}
\maketitle


\section{Statement of result}

 In this paper, we prove a new regularity criteria, beyond type I condition, for Leray-Hopf solutions to  the axially symmetric Navier-Stokes equations in $\mathbf{R}^3 \times (0, \infty)$, abbreviated as ASNS henceforth:
\be
\begin{aligned}
\label{eqasns}
\begin{cases}
   \big (\Delta-\frac{1}{r^2} \big )
v_r-(v_r \p_r + v_3 \p_{x_3})v_r+\frac{(v_{\theta})^2}{r}-\partial_r
P-\p_t  v_r=0,  \\
   \big   (\Delta-\frac{1}{r^2}  \big
)v_{\theta}-(v_r \p_r + v_3 \p_{x_3} )v_{\theta}-\frac{v_{\theta} v_r}{r}-
\partial_t v_{\theta}=0,\\
 \Delta v_3-(v_r \p_r + v_3 \p_{x_3})v_3-\p_{x_3} P-\p_t v_3=0,\\
 \frac{1}{r} \p_r (rv_r) +\p_{x_3}
v_3=0,\\
v(\cdot, 0)=v_0(\cdot).
\end{cases}
\end{aligned}
\ee
Here, $P$ is the pressure,  $v = v_{r}e_{r} + v_{\th}e_{\th} + v_{3}e_{3}$ is the velocity in the cylindrical system with the standard basis $\{e_{r}, e_{\th}, e_{3}\}$, where for any $ x=(x_1,x_2,x_3)\in \mathbf{R}^3 $, $ r=\sqrt{x_1^2+x_2^2} \equiv |x'|$ and
\be\label{ert3}
e_r=(x_1/r, x_2/r, 0),\quad e_\th=(-x_2/r, x_1/r, 0),\quad e_3=(0,0,1).\ee
The components $v_{r}$, $v_{\th}$ and $v_{3}$ are independent of the azimuthal angle $\th$. Although ASNS is a special case of the full 3D Navier-Stokes equations,
\be
\label{nse}
\Delta v -  (v\cdot \nabla) v - \nabla P -\partial_t v =0, \quad \text{div} \, v=0, \quad \mathbf{R}^3 \times (0, \infty), \, v(x, 0)=v_0,
\ee
the regularity problem of the former is still open in general. It is also referred to as the $2 \frac{1}2$ dimensional regularity problem.
Let us recall that Leray \cite{Le2}, proved in the 1930s that the Cauchy
  problem of \eqref{nse} has a solution in the energy space $(v, \nabla v) \in   (L^\infty_t L^2_x,
  L^2_{x t})$. However, in general it is not known if such solutions stay bounded or
  regular for all $t>0$ given regular initial values.
Over the years, many  researchers have spent effort on the regularity problem of ASNS,  a small list of which include the relevant papers \cite{La, UY, CSTY1, CSTY2, KNSS, CFZ, LZ17, Weid} and the references therein. In particular, in the swirl free case ($v_\theta=0$), the regularity problem was solved in \cite{La, UY}.

 From the 1960s, it is already realized by Ladyzhenskaya and Yudovic that one key quantity for the ASNS is $\Gamma=r v_\theta$ which is scaling invariant and which satisfies the drift diffusion equation:
\be
\lab{eqgam1}
\begin{cases}
\al
&(\p^2_r + \p^2_{x_3})  \Gamma  - \frac{1}{r} \p_r  \Gamma  -(v_r, v_3) \cdot (\p_r, \p_{x_3})  \Gamma  - \p_t \Gamma=0, \quad r \ge 0, t \in (0, T),\\
&\Gamma(0, t)=0, \quad t \ge 0;  \quad \Gamma(r, 0)= \Gamma_0(r).
\eal
\end{cases}
\ee Therefore $\Gamma$ satisfies the maximum principle, giving us the a priori bound
$|\Gamma(x, t)| \le \Vert \Gamma(\cdot, 0) \Vert_\infty$, i.e. the swirl velocity satisfies $|v_\theta(x, t) | \le \Vert \Gamma(\cdot, 0) \Vert_\infty/|x'|$. As before $x'=(x_1, x_2, 0)$ if $x=(x_1, x_2, x_3) \in \mathbf{R}^3$.
Notice that this a priori bound for $v_\theta$ is already critical instead of super critical.  One would hope to prove that $\Gamma$ has some regularity near the axis $x'=0$, so that $v_\theta(x, t)$ would become a little less singular. However, the drift terms $(v_r, v_3)$ being in  the usual energy space, are still super critical and far too singular to apply the classical De Giorgi-Nash-Moser theory. In the years 2008-2009, two groups of authors Chen, Strain,  Tsai and  Yau \cite{CSTY1, CSTY2} and respectively Koch, Nadirashvili,  Seregin and Sverak \cite{KNSS} essentially proved $\Gamma$ is H\"older continuous under certain critical assumptions on $v$, i.e. $|v(x, t)|<C/|x'|$ or $|v(x, t)| \le C/\sqrt{T-t}$. This is enough to prove regularity for Leray-Hopf solutions. Solutions satisfying these kind of scaling invariant conditions are often referred to as type I solutions. So the remaining cases to deal with are type II (non type I or supercritical) solutions. In \cite{Px} a slightly super-critical regularity condition is obtained.

Around 2015, Chen-Fang-Zhang \cite{CFZ} realized that H\"older continuity of $\Gamma$ alone is enough to prove regularity. Their idea is to work on the system of equations for the modified vorticity $\Omega =\omega_\theta/r$ (known since 1960s \cite{La, UY}) and $J=\omega_r/r$: for $b=v_r e_r + v_3 e_3$,
\begin{equation}
\label{eqjoo}
\begin{cases}
\Delta J  -(b\cdot\nabla) J +\frac{2}{r}\p_r J +
 (\omega_r \p_r + \omega_3 \p_3) \frac{v_r}{r} - \p_t J
=0,\\
\Delta \Omega -(b\cdot\nabla)\Omega+\frac{2}{r}\p_r
\Omega - \frac{2v_{\theta}}{r} J -\p_t \Omega=0.\\
\end{cases}
\end{equation}
Shortly after, it was realized in \cite{LZ17} by Lei and the author that the vortex stretching terms in this modified system is critical instead of super critical. Moreover, a $1/|\ln r|^2$ modulus continuity alone for $\Gamma$ is sufficient for the regularity of the whole solution. This was improved to $1/|\ln r|^{3/2}$ modulus continuity by Wei \cite{Weid} a little later.These results induce some expectation that the regularity problem in this case is becoming accessible one way or the other. However the drift terms $b$ are still super critical and we ran into the same old difficulty as before.

The following is the main result of the paper. Since only one sided bound in the radial velocity is needed, it provides a regularity criteria that goes beyond the type I condition. We call this condition partial type I condition.

\begin{theorem}
\lab{thmain}
Let $v=v(x, t)$ be a Leray-Hopf solution of the axially symmetric Navier-Stokes equations in $\mathbf{R^3} \times [0, \infty)$. Suppose, (a). the initial value $v_0=v(\cdot, 0)$ is in the space
$L^2(\mathbf{R^3}) \cap L^\infty(\mathbf{R^3}) \cap C^3(\mathbf{R^3})$,  (b). $r v_\theta(\cdot, 0) \in L^\infty(\mathbf{R^3})$,  (c). $v$ is partially type I, i.e.:  $v_r(x, t) \ge -c/\sqrt{T-t}$, $(x, t) \in \mathbf{R}^3 \times [0, T)$ for any given constants $c>0$ and $T>0$. Then $v$ is bounded and smooth in $\mathbf{R^3} \times [0, T]$.
\end{theorem}

Let us explain the idea of the proof.
The new input is to construct a modulus of continuity $u=u(r, t)$ for $\Gamma=r v_\theta$, which satisfies the one dimensional drift diffusion equation in the half line:
\be
\lab{equro0}
\begin{cases}
\p^2_r u(r, t) + g(t) \p_r u(r, t) - \p_t u(r, t) = 0, \qquad (r, t) \in [0, \infty) \times [0, T),\\
u(0, t)=0, \quad u(r, 0)=u_0(r),
\end{cases}
\ee where $g=g(t)$ is in $C([0, T))$ and $g(t)>0$ for each $t$. We choose a $g(t) \ge \Vert v^-_r(\cdot, t) \Vert_\infty$ which is bounded and smooth before the first possible blow up time, at which moment it can be very singular.  It is known from \cite{SVZ} and \cite{W} that modulus of continuity is not preserved for the heat equation with some slightly super critical drifts. Are we running into the old trap?  It turns out that when the drift term depends only in time, the situation is different.   One hint of this gain of the spatial regularity is the equation on the whole line
\be
\lab{equro00}
\begin{cases}
\p^2_r u(r, t) + g(t) \p_r u(r, t) - \p_t u(r, t) = 0, \qquad (r, t) \in \mathbf{R}^1 \times [0, T),\\
 u(r, 0)=u_0(r).
\end{cases}
\ee The explicit solution is
\[
u(r, t) = \frac{1}{\sqrt{4 \pi t}} \int \exp \left(- \big|r - y + \int^t_0 g(s) ds \big|^2/(4 t)\right) u_0(y) dy.
\]We see that $\p_r u(\cdot, t)$ is bounded for $t>0$ independent of $g$ if $u_0 \in L^1$. But there is no such gain in regularity in the time direction.

Of course the half line case is different since the boundary value may interfere. Our job is to analyze this situation carefully. It turns out that after the above change of variable, one can turn \eqref{equro0}  into an initial Dirichlet problem for the standard heat equation on a time dependent domain. Using the standard boundary regularity theorem, which requires only the so called exterior measure condition, on can prove that $u$ is H\"older continuous under the one sided condition $g(t) \ge -c/\sqrt{T-t}$. The point is that only one sided critical condition is required on $v_r$. 

Next, based on the H\"older continuity of $u$, one uses comparison theorem (maximum principle), to deduce that $|r v_\theta (x, t)| = |\Gamma(x, t)| \le C |x'|^\alpha$. This is enough to prove regularity of \eqref{eqasns} by \cite{CFZ} mentioned above.
We will also give an example that the condition on $g=g(t)$ can not be improved by a log factor as far as regularity for $u$ is concerned. So we are still quite far from the a priori estimate that $\Vert v_r(\cdot, t) \Vert_\infty \in L^1[0, T]$, c.f. \cite{FGT} and \cite{Cons14}. Since a good regularity criteria can also help to construct potential blow up solutions by narrowing the pool of the candidate, such a result is potentially useful one way or the other. This result confirms the intuition that possible blow up is caused by rapid inward flows. Let us mention that in \cite{Px2} and \cite{Zz},  related regularity conditions $v_r(x, t) \ge -M/r$ were proven for the ASNS. But the critical lower bound there has a size restriction of $M=1$ and $M \in (1, 2)$ respectively.

\section{ Proof of  theorem \ref{thmain} }

Our starting point is the following lemma on the H\"older continuity of solutions of the one dimensional heat equation with a divergence free drift $g(t)=C/\sqrt{T-t}$ in the half line.
Let us recall that the standard De Giorgi-Nash-Moser space-time H\"older regularity result on $u$ requires $g$ being in $L^{2^+}_{loc}$.

\begin{lemma}
\lab{leduk} Suppose $g =K/\sqrt{T-t}$ where $K$ is any positive constant.
 The one dimensional heat equation \eqref{equro0} with drift $g$ in the half line
\be
\lab{eqku}
\begin{cases}
\p^2_\r u(\r, t) + g(t) \p_\r u(\r, t) - \p_t u(\r, t) = 0, \qquad (\r, t) \in [0, \infty) \times [0, T),\\
u(0, t)=0, \quad u(\r, 0)=\r,
\end{cases}
\ee has  a nonnegative, smooth solution $u$ with the following properties.

(a). There exist constants $C_0 \ge 1, \alpha \in (0, 1)$, depending only on $K$ and $T$, such that
\[
0 \le  u(\r, t) \le C_0 (\r^\alpha + \r ),   \quad \r \ge 0, \,  t \in [0, T).
\]

(b). There exist positive constants $\r_1=\r_1(K, T)>1$ and $C_1=C_1(K, T)$, such that
\[
\p_\rho u(\r, t) \le C_1, \quad \text{when} \quad \r \ge \r_1, \, t \in [0, T).
\]

(c). For every $\e>0$, there exists a constant $\r_0=\r_0(\e, K, T)>1$, such that
\[
u(\r, t) \ge (1-\e) \r, \quad \text{when} \quad \r \ge \r_0, \, t \in [0, T).
\]

\proof
\end{lemma}

 The uniqueness of $u$ follows from the maximum principle in classical theory since all data involved is smooth for $t <T$  and the initial value is linear, far below the required condition of exponential growth of order 2. This property will be used several times in the paper. The maximum principle also implies that $\p_\r u \ge 0$. A more direct proof is to make a change of variables as in \eqref{z=rhoa} below to eliminate the drift term and use $\epsilon (\rho^2+ 2 t)$ as a barrier function, then let $\epsilon \to 0$.

 For the existence, let
 $H=H(\r, t, \tilde{\r}, s)$ be the Dirichlet heat kernel in the half line. We can solve the integral equation
 \[
 u(\r, t) = \int^\infty_0 H(\r, t, \tilde{\r}, 0) \tilde{\r} d\tilde{\r} +
 \int^t_0 \int^\infty_0 H(\r, t, \tilde{\r}, s) g(s) \p_{\tilde{\r}} u(\tilde{\r}, s)
 d\tilde{\r} ds.
 \]Since $g \in C^\infty[0, T-\e]$, $\e>0$ then one can iterate the integral equation to obtain a solution $u$ of linear growth in the time interval $[0, T-\e]$.  Moreover one arrive at a $\Vert \p_\r u \Vert_\infty $ bound depends on $\Vert g \Vert_\infty$ which is finite in the time interval $[0, T-\e]$, $\e>0$ but which may blow up at time $T$.

Alternatively, for the existence, one can also solve the following Robin boundary value problem with initial value $1$, which is standard for $t \in [0, T)$:
\be
\lab{eqkurob}
\begin{cases}
\p^2_r v(r, t) + g(t) \p_r v(r, t) - \p_t v(r, t) = 0, \qquad (r, t) \in [0, \infty) \times [0, T),\\
\p_\r v(0, t) + g(t) v(0, t)=0, \quad v(r, 0)=1.
\end{cases}
\ee It is easy to see that $u(\r, t) \equiv \int^\r_0 v(l, t) dl$ is a solution to \eqref{eqku}.

To prove property (a),
let us write
\be
\lab{at=}
A(t) \equiv \int^t_0 g(s) ds = 2 K (\sqrt{T}-\sqrt{T-t}), \quad t<T.
\ee
Consider the spatial variable
\be
 \lab{z=rhoa}
 z=\rho + A(t)
\ee and the function
 \be
 \lab{nu=u}
 \nu=\nu(z, t)=\nu(\rho + A(t), t) \equiv u(\rho, t).
 \ee Then the domain for $\nu$ is
 \be
 \lab{domP}
 P \equiv \{ (z, t) \, | \, z > A(t), \quad t \in [0, T) \} \subset \mathbf{R}^2_+.
 \ee Since
\[
\p^2_\rho u(\rho, t)= \p^2_z \nu(z, t), \quad \p_t u(\rho, t) = \p_t \nu(z, t) + \p_z \nu(z, t) g(t),
\]we find
\[
\al
\p^2_\rho u(\rho, t) - \p_t u(\rho, t) &= \p^2_z \nu(z, t) - \p_t \nu(z, t) - \p_z \nu(z, t) g(t)\\
&=\p^2_z \nu(z, t) - \p_t \nu(z, t) - \p_\rho u(\rho, t) g(t).
\eal
\]This and \eqref{eqku} imply that $\nu=\nu(z, t)$ is a local $C^{2, 1}_{z, t}$ solution to the following one dimensional heat equation with Dirichlet boundary value in the time-varying domain $P$:
\be
\lab{equz0}
\begin{cases}
\p^2_z \nu(z, t) - \p_t \nu(z, t) = 0, \qquad (z, t) \in P,\\
\nu(z, t)=0, \,  z=  A(t), \quad t \ge 0;\\
  \nu(z, 0)=u_0(z)=z.
\end{cases}
\ee Notice that $f=z$ is a solution of the heat equation with the same initial value as $\nu$ and it is non-negative on the boundary $z=A(t)$. The maximum principle, which works for solutions of linear growth,  implies
\be
\lab{nu<z}
0 \le \nu(z, t) \le z
\ee and therefore
\[
u(\r, t) \le z = \r + A(t) \le \r + 2 K \sqrt{T}.
\]This yields the stated bound in part (a) when $\r \ge  2 K \sqrt{T}$.

Next we consider the remaining case when $0 < \rho \le 2 K \sqrt{T}$.

Notice that the left boundary curve $z=A(t)=\int^t_0 g(s) ds$ is  a parabolic boundary of the domain $P$ for $t \in (0, T)$.  Fixing a small $\delta>0$,  for any $r \in (0, \delta]$ and any boundary point $(A(t), t), t \in [\delta^2, T]$, we have
\be
\lab{q012}
|Q_2(A(t), t, r) \cap P^c | \ge c_0 |Q_2(A(t), t, r)|
\ee where $c_0$ is a positive constant depending only on $K$ and
\[
Q_2=Q_2(z, t, r) = \{ (y, s) \, | \, |y-z|<r, \quad t-r^2<s<t \}
\] is the standard parabolic cube in $\mathbf{R}^2$. This property just means that parabolic cubes whose "centers" straddle the parabolic  boundary of $P$ must intersect the exterior domain $P^c$ by at least a definite portion. It is here that we use crucially $\eqref{at=}$, which is scaling invariant under the standard parabolic scaling.

Therefore, the domain $P$ satisfies Condition A, i.e. (6.46) p133 in the book \cite{Lieb}. This condition is  also referred to as exterior measure condition. We mention the typo there that $P$ was written in Condition A instead of $P^c$, but as well known the proof requires $Q_2 \cap P^c$ takes a definite portion of the parabolic cube $Q_2$.  Now we can apply Theorem 6.32 in the said book to conclude that for some $\alpha \in (0, 1)$ and $C_*>0$:
\be
\lab{0uhold}
\Vert \nu \Vert_{C^{\alpha, \alpha/2}(Q_2(z, t, \delta) \cap P)} \le \frac{C_*}{\delta^\alpha}
\Vert \nu \Vert_{L^\infty(Q_2(z, t, 2 \delta) \cap P)}, \quad (z, t) \in P, \quad t > \delta^2.
\ee Let us explain why $\alpha \in (0, 1)$ and $C_*>0$ are absolute constants. First, we mention that in the cited theorem, the boundary of the domain is assumed to be locally $C^1$, which is true for $t<T$. However, the constants $\alpha$ and $C_*$ are independent of the $C^1$ norm of the boundary since the proof relies only on the weak parabolic Harnack inequality in the De Georgi-Nash-Moser theory for the solutions after $0$ extension beyond the boundary. In fact they only depend on the constant in (6.46) in the cited book, which is $c_0$ here and the coefficient in the elliptic part of the equation, which is $1$ here. Note also that there are no lower order terms in the equation. Another proof with a different method for \eqref{0uhold} can be found in \cite{FS}, which is further developed in \cite{CDK}.

Since
\[
\nu(\rho + A(t), t) \equiv u(\rho, t),
\]we deduce
\[
u(\r, t) \le C_0 \r^\alpha, \quad \text{if} \quad \r \le 2 K \sqrt{T}, \,  t \ge \delta^2.
\]Lastly, if $\r \le 2 K \sqrt{T}$ and $t \in [0, \delta^2]$, the property is a standard boundary regularity result for the heat equation, since $b=b(t)$ here is smooth and bounded and the initial value is compatible.
This proves (a).

 Next we prove (b). For $z>A(t)+1$, we can use \eqref{equz0} and \eqref{nu<z} and standard interior gradient estimate of the heat equation on  small parabolic cubes of fixed size to deduce
\[
0 \le \p_z \nu(z, t) \le C z, \quad z>A(t)+1.
\]Here $C$ is an absolute constant. Note $\p_\rho \nu(z, 0)=1$ and $0 \le \p_\rho \nu(A(t)+1, t) \le C(A(T)+1)$. Hence we can apply the maximum principle on the space time domain $\{(z, t) \, | \, z>A(t)+1, t \in [0, T) \}$ for $\p_\rho \nu$, which is a solution of the heat equation of linear growth. This infers
\[
\p_z \nu \le \max\{1, C(A(T)+1)\},
\]
which proves statement (b) since $\p_z \nu(z, t) = \p_\r u(\r, t)$ by \eqref{nu=u}.

\medskip

For property (c), it is sufficient to prove the following statement for $\p_\r u$:
for every $\e>0$ and $T>1$, there exists a constant $\r_0=\r_0(\e, k, T)>1$, such that
\be
\lab{du=1}
\p_\r u(\r, t) \ge (1-\e), \quad \text{when} \quad \r \ge \r_0, \, t \in [0, T].
\ee If this is false, there exists one small $\e_0>0$, and sequences $t_j \in [0, T]$ and $\r_j \to \infty$ such that
\be
\lab{rhoduj}
\lim_{j \to \infty} \p_\r u(\r_j, t_j) \le 1-\e_0.
\ee Without loss of generality, we can assume $t_j \to t_* \in (0, T]$. Consider the shifted functions
\[
h_j=h_j(l, t)= \p_l u(l-A(t)+\r_j, t), \quad l \ge -\r_j/2,   \, t \in [0, T].
\]By part (b), $h_j$ are uniformly bounded when $l \ge -\r_j/2$. Note $h_j$ is a solution of the heat equation and $h_j(\cdot, 0)=1$. From the standard regularity theory of, we know that a subsequence of $\{h_j \}$ converges in $C^{2, 1}_{x, t, loc}$ norm, to a bounded, smooth solution $h$ of the heat equation:
\[
\begin{cases}
\p^2_l h(l, t)  - \p_t h(l, t) =0, \quad (l, t) \in \mathbf{R}^1 \times [0, T]\\
h(l, 0)=1>0.
\end{cases}
\] But $h(A(t_*), t_*) \le 1-\e_0 <1$ due to \eqref{rhoduj}, which violates the maximum principle. Hence \eqref{du=1} is true. This proves statement (c) and the lemma.
\qed

\medskip

 One consequence of the above lemma is the following result which will be used in proving the main theorem.

\begin{lemma}
\lab{lem1ddifs1}
Let $g=g(t)>0$ be as in the previous lemma. Let $u_0$ be a function in $C^3([0, \infty))$ such that $0 \le u_0(\r) \le \alpha_0 \r$ for a positive constant $\alpha_0$.
The following initial boundary value problem with $g$ as drift and $u_0$ as the initial value
\be
\lab{equro2}
\begin{cases}
\p^2_\rho u(\rho, t) + g(t) \p_\rho u(\rho, t) - \p_t u(\rho, t) = 0, \qquad (\rho, t) \in [0, \infty) \times [0, T),\\
u(0, t)=0, \quad u(\rho, 0)=u_0(\rho),
\end{cases}
\ee has a local $C^{2, 1}_{\rho, t}$ solution $u=u(\rho, t) \ge 0$ with the following properties.

(a). There exist positive constants $C_0$ and $\alpha \in (0, 1)$, which depends only on the constant $K$ in $g$ such that
\be
\lab{ud0jie}
0 \le u(\r, t) \le C_0 \alpha_0 (\r^\alpha +\r), \quad \qquad (\rho, t) \in [0, \infty) \times [0, T).
\ee

(b). Suppose in addition that $u_0(\r) = \alpha_0 \r$ when $\r \ge 10$ . For every $\e>0$, there exists a constant
$\r_0=\r_0(K, \e)>1$, such that
\[
u(\r, t) \ge (1-\e) \alpha_0 \rho, \quad \text{when} \quad \r \ge \r_0, \, t \in [0, T).
\]

(c). Moreover,
\be
\lab{ud2jie}
0 \le \p_\r u(\r, t) \le C_0 \alpha_0,  \quad \qquad (\rho, t) \in [\r_0, \infty) \times [\delta, T)
\ee
\proof
\end{lemma}

(a).

From the assumption, there exist a positive constant $\alpha_0$, depending only on $u_0$ such that
\[
u_0(\r) \le \alpha_0 \r, \quad \r \ge 0.
\]Let $u_1 \ge 0$ be the solution, given in the previous lemma,  to the problem
\be
\lab{equrok}
\begin{cases}
\p^2_\rho u_1(\rho, t) + g(t) \p_\rho u_1(\rho, t) - \p_t u_1(\rho, t) = 0, \qquad (\rho, t) \in [0, \infty) \times [0, T),\\
u_1(0, t)=0, \quad u_1(\rho, 0)=\alpha_0 \r.
\end{cases}
\ee

For any large integer $j$, let $h_j$ be the solution of the initial-boundary value problem in the bounded domain $[0, j]$:
\be
\lab{equro2i}
\begin{cases}
\p^2_\rho h_j(\rho, t) + g(t) \p_\rho h_j(\rho, t) - \p_t h_j(\rho, t) = 0, \qquad (\rho, t) \in [0, j] \times [0, T),\\
h_j(0, t)=0, \quad h_j(j, t) = u_1(j, t), \quad u(\rho, 0)=u_0(\rho),
\end{cases}
\ee

 Subtraction of this by \eqref{equrok} yields, for $(\rho, t) \in [0, j] \times [0, T)$, that
\be
\lab{equrou-uk}
\begin{cases}
\p^2_\rho (h_j- u_1)(\rho, t) + g(t) \p_\rho (h_j-u_1)(\rho, t) - \p_t (h_j-u_1)(\rho, t)
=  0,  \\
(h_j-u_1)(0, t)=(h_j-u_1)(j, t) =0, \quad (h_j-u_1)(\rho, 0)=u_0(\r)-\alpha_0 \r \le 0 .
\end{cases}
\ee  The basic maximum principle on bounded domains infers that
\[
0 \le h_j(\r, t) \le u_1(\r, t), \quad (\rho, t) \in [0, j] \times [0, T).
\]Now Lemma \ref{leduk} tells us that there exists a constant $C_0 \alpha_0$,  that
\[
0 \le h_j(\r, t) \le u_1(\r, t) \le C_0 \alpha_0(\r^\alpha+ \r), \quad (\rho, t) \in [0, j] \times [\delta, T).
\] Note that $\alpha_0=1$ in the statement of that lemma and the linearity of the problem allows us to multiply it back.

By standard Schauder regularity theory for parabolic equations (c.f. \cite{Lieb} Ch. IX), a subsequence of $\{ h_j \}$ converges in local $C^{2, 1}_{x, t}$ sense, to a solution $u$ of \eqref{equro2} in $[0, \infty) \times [0, T)$ and $u$ satisfies
\be
\lab{udeltjie}
0 \le u \le C_0 \alpha_0 (\r^\alpha + \r), \quad (\rho, t) \in [0, \infty) \times [\delta, T).
\ee
This and uniqueness proves  \eqref{ud0jie}.

(b). This is identical to the proof of part (c) of the previous lemma.

(c).
This follows from part (b) of the previous lemma.
\qed
\medskip

Now we are in a position to give a proof of the main theorem.

\proof  {\it Step 1.}  Let $v=v(x, t)$ be a Leray-Hopf solution of the ASNS such that $v_0 \in L^2 \cap L^\infty \cap C^3$. It is well known that $v$ is smooth at least in a short time interval. So we can suppose that $v$ is smooth in the time interval $(0, T)$. We will show that it is smooth in $(0, T]$.   Recall that for ASNS, the scaling invariant quantity $\Gamma = r v_\th$ satisfies the equation without the pressure term:
\be
\lab{eqgam2}
\begin{cases}
\al
&(\p^2_r + \p^2_{x_3})  \Gamma  - \frac{1}{r} \p_r  \Gamma  -(v_r, v_3) \cdot (\p_r, \p_{x_3})  \Gamma  - \p_t \Gamma=0, \quad r \ge 0, t \in (0, T),\\
&\Gamma(0, t)=0, \quad t \ge 0;  \quad \Gamma(r, 0)= \Gamma_0(r), \quad |\Gamma_0(r)| \le \alpha_0 r.
\eal
\end{cases}
\ee Here $\alpha_0 = \Vert v_\th(\cdot, 0) \Vert_\infty$. In addition $\p_r \Gamma(0, t)=0, \, t<T$.

The modulus of continuity we look for is a $C^{2, 1}_{r, t}$ function $\La=\La(r, t)$ of $r, t$ only, which satisfies the equation
\be
\lab{eqLa}
\begin{cases}
\p^2_r \La(r, t) - \frac{1}{r} \p_r \La(r, t) + g(t) \p_r \La(r, t) - \p_t \La(r, t) = 0, \qquad (r, t) \in (0, \infty) \times [0, T),\\
\La(0, t)=0, \quad t \ge 0;  \\
\La(r, 0)= \La_0(r) \ge |\Gamma_0(r, x_3)|,
\end{cases}
\ee where $g=g(t)$ is a smooth function on $[0, T)$ such that
\be
\lab{gt>vr}
 g(t) \ge  \Vert v^-_r(\cdot, t) \Vert_\infty
\ee and
\[
\La_0(r)=
\begin{cases}
 \alpha_0 r^2, \quad r \le 1/2, \\
\text{smooth and monotone} \quad 1/2<r<1, \, 0<\La'_0 \le 2 \alpha_0,\\
\alpha_0 r, \quad  r \ge 1.
\end{cases}
\]We observe that $v_\th(\cdot, 0)$ always vanishes at $r=0$ so that $\Gamma(\cdot, 0)$ vanishes at $r=0$ at least order $2$. So the initial value for $\La$ can be realized.
By the assumption on $v_r$, we can take, for some constant $K>0$,
\be
\lab{gin1}
g =K/\sqrt{T-t},
\ee which allows us to use Lemma \ref{lem1ddifs1} later.

In the next step, we prove that such a function $\La$ with the property that $\p_r \La \ge 0$ can be found.
\medskip

{\it Step 2.}

 {\it step 2.1.} Existence and regularity of $\La$.

  The existence of a generic solution to \eqref{eqLa} is expected since $g$ is bounded and smooth for $t \in [0, T-\e]$, $\e>0$ so that we can just work on $[0, T-\e]$ first and then let $ \e \to 0$. So we will always assume $t \in [0, T-\e]$ in this sub-step. Note all the bounds in this sub-step may blow up when $\e \to 0$.

 However, there is the singularity $-1/r$ in the drift term in \eqref{eqLa}, which,  although being mild,  is still a critical one. Therefore certain care is needed in the proof which we start now.
It is easy to check that $\La$ is a solution to \eqref{eqLa} if and only if the function
\[
f=\La /r=\La(r, t)/r
\] is a solution of the  diffusion equation of variable $(r, t)$ in the cylindrical coordinates:
\be
\lab{eqrlam}
\Delta f - \frac{1}{r^2} f + g(t) \p_r f + \frac{g(t)}{r} f -\p_t f=0,   \quad f(r, 0)=\La(r, 0)/r.
\ee Here $\Delta=\p^2_r + \frac{1}{r} \p_r$ is the radial Laplacian in $\mathbf{R}^2$. Note that the initial value $f(r, 0)$ is bounded. So it suffices to prove the existence and regularity of $f$. Using the property that the singularity of the potential is the standard inverse square with  a good sign, the potential $\frac{1}{r}$ is of order $-1$, hence sub-critical,  and that the initial value is smooth and bounded, we will show that this equation has a  bounded solution which is smooth for $r>0$ and $t<T$, and which vanishes at $r=0$. Once this is done, we will know that  $\La= r f$ is also smooth for $r>0$ and vanishes at order $2$ at $r=0$.

We will take the following sub-steps in proving the  above assertion.
1. derive the equations for $ f \cos \theta$ and $f \sin \theta$.  2. convert these equations to integral equations \eqref{eqfi}. 3. solving the integral equations by iteration to obtain H\"older continuous solutions $f$. 4. prove further that $\p_r f$ is H\"older continuous. 5. existence and regularity of $\La$ follows from 4 since $\La= rh$.

So, we start by considering the two functions
\be
\lab{f1f2sin}
f_1=f(r, t) \cos \theta, \qquad f_2=f(r, t) \sin \theta.
\ee Using \eqref{eqrlam}, direct computation shows, for $i=1, 2$, that, for $x, y \in \mathbf{R}^2$, $t \in [0, T)$, $r=|x|, \, x=(r \cos \theta, r \sin \theta)$, the following simpler equations hold:
\be
\lab{fipde}
\Delta f_i(x, t) + g(t) \p_{|x|} f_i(x, t) + \frac{g(t)}{r} f_i - \p_t f_i(x, t)=0,
\ee and hence
\be
\lab{eqfi}
\al
f_i(x, t)&=\int G_0(x, t; y, 0) f_i(y, 0) dy + \int^t_0 \int G_0(x, t; y, s) \frac{g(s)}{|y|} \p_{|y|} ( f_i(y, s) |y|)  dyds.
\eal
\ee Here $G_0$ is the standard heat kernel in $\mathbf{R}^2$. After integration by parts in the polar coordinates, we deduce:
\be
\lab{eqfiint}
\al
f_i(x, t)&=\int G_0(x, t; y, 0) f_i(y, 0) dy
 - \int^t_0 \int \p_{|y|} G_0(x, t; y, s)  g(s) f_i(y, s) dyds.
\eal
\ee
Since $g$ is bounded and smooth in $[0, T-\e]$ and the kernels are integrable in space time, by standard iteration, we see that there exists bounded solutions for $f_i$. Indeed, for a small $t_1$ and $t \in [0, t_1]$ and $f_i \in L^\infty(\mathbb{R}^2 \times [0, T_1])$,  one can consider the integral operator $M$:
\[
 (M f_i)(x, t) =\int G_0(x, t; y, 0) f_i(y, 0) dy
 - \int^t_0 \int \p_{|y|} G_0(x, t; y, s)  g(s) f_i(y, s) dyds.
\] Using the contraction mapping principle, one obtains a fixed point for $M$, which solves \eqref{eqfiint}. Next, one can use $t_1$ as the new initial value and repeat this process to another time $t_2>t_1$ and so on. Since the equation is linear, this can be repeated before the given time $T>0$.

The function $f_i$ obtained in the previous paragraph are weak solutions of \eqref{fipde}, whose coefficients are smooth except at $r=0$.  It is easy to see that $f_i$ are smooth except when $r=0$ and $f_i$ are H\"older continuous through out. For the latter, see also the comments before \eqref{eqhf0} below. Since $f=e^{-i \theta} (f_1 + i f_2)$, we see that $f$ is also bounded. This  implies that $f_i$ vanishes when $r=0$ by taking $r \to 0$ along the line $\theta=\pi/2$ and $\theta=0$ and using $\cos (\pi/2)=0$ and $\sin 0=0$ in \eqref{f1f2sin}. Hence $f(0, t)=0$. This process gives a proof of the existence of $f$.
Since
\be
\lab{lamf12}
\La= r f = r e^{-i \theta} (f_1 + i f_2),
\ee this proves the existence of $\La$.
 Notice that, when $r >0$ and $t<T$, all coefficients in the equation are locally bounded and smooth, so the standard regularity theory implies that $\nabla f$ is locally bounded. We mention that this bound may depend on $r>0$, $g$ and $t$ and is not uniform yet.
 In the sequel, we will need to use the maximum principle on $\La$ or $\p_r \La$ so we need to prove $\La$ and $\p_r \La$ are smooth in the interior and continuous up to the boundary. The former is already proven since $f_i$ are smooth in the interior.  Next we prove that $\p_r \La$ is continuous up to the boundary.
From \eqref{lamf12} and using the property that $\La$ depends only on $r$ and $t$ but not $\theta$, it is sufficient to prove $\p_r f_i$ are H\"older continuous in $r$ through out.

To this end, we rewrite \eqref{eqfi} as
\be
\lab{eqfi2}
\al
f_i(x, t)&=\int G_0(x, t; y, 0) f_i(y, 0) dy + \int^t_0 \int G_0(x, t; y, s) g(s) \p_{|y|}  f_i(y, s)  dyds\\
&\qquad + \int^t_0 \int G_0(x, t; y, s) \frac{g(s)}{|y|} f_i(y, s)  dyds,
\eal
\ee which implies
\be
\lab{eqfi3}
\al
\p_{|x|} f_i(x, t)&=\underbrace{\int \p_{|x|} G_0(x, t; y, 0) f_i(y, 0) dy}_{I_1} + \underbrace{ \int^t_0 \int \p_{|x|} G_0(x, t; y, s) g(s) \p_{|y|}  f_i(y, s)  dyds}_{I_2}\\
&\qquad + \underbrace{\int^t_0 \int \p_{|x|} G_0(x, t; y, s) \frac{g(s)}{|y|} f_i(y, s)  dyds }_{I_3}.
\eal
\ee Therefore
\be
\lab{eqdfi}
\al
|\p_{|x|} f_i(x, t)| &\le \underbrace{ \int  G_0(x, t; y, 0) |\nabla_y f_i(y, 0)| dy
+ C \int^t_0 \int |\p_{|x|} G_0(x, t; y, s)| \,  \frac{g(s)}{|y|^{1-\alpha}}   dyds}_{K(x, t)} \\
&+ \int^t_0 \int |\p_{|x|} G_0(x, t; y, s)|  \, |\p_{|y|} f_i(y, s)| g(s) dyds.
\eal
\ee Here we have used the property that $f_i$ are H\"older continuous with parameter $\alpha>0$.

By the choice of the initial value and the standard bound
\[
|\p_{|x|} G_0(x, t; y, s)| \le \frac{C}{(t-s)^{3/2}} e^{-\frac{|x-y|^2}{5(t-s)}},
\]both terms in $K(x, t)$ on the right hand side is bounded.
For fixed $\e>0$, recall that $g(t)$ is bounded  for $t \in [0, T-\e]$. Notice also that
$|\p_{|x|} G_0(x, t; y, s)|$ is a space-time integrable kernel. So by standard iteration
we deduce
\be
\lab{dfijie}
|\p_{|x|} f_i(x, t)| \le C(t, \Vert g \Vert_{L^\infty[0, T-\e]}),
\ee which shows $\p_r \La$ is bounded by \eqref{lamf12}. With \eqref{dfijie} at hand, it is straight forward to check that the terms $I_1$, $I_2$ and $I_3$ in \eqref{eqfi3} are H\"older continuous. Indeed the integrand beside the kernel in $I_1$ is smooth and the integrands besides the kernel  in $I_2$ and $I_3$ are now bounded. Recall now $|f_i(y, s)|/|y| \le C$ due to $f_i(0, s)=0$ and \eqref{dfijie}. Hence $\p_{|x|} f_i$ and consequently $\p_r \La$ is H\"older continuous.  Another way to see it is to consider the nonhomogeneous heat equation
\be
\lab{eqhf0}
\begin{cases}
\Delta H- \p_t H= F, \quad \mathbf{R}^n \times [0, \infty),\\
H(\cdot, 0)=0.
\end{cases}
\ee By direct calculation using the heat kernel or using standard theory, c.f. \cite{Lieb} Theorem 6.29, p130, if $H$ is a solution which is locally in the energy space and $F \in L^\infty$, then $\nabla H$ is H\"older continuous, since $\Delta \nabla  H- \p_t \nabla H= \nabla  F$ with $F \in L^\infty$.

\medskip

{\it step 2.2.} Monotonicity of $\La(\cdot, t)$.

 So it remains to check $\p_r \La(r, t) \ge 0$, using the boundary and initial conditions: $\p_r \La(0, t) \ge 0$ and $\p_r \La(r, 0) \ge 0$  for $r>0$ and $t>0$. The main issue is to deal with the singular terms in the equation for $\p_r \La(r, t)$.

{\it step 2.2.1.}

 Note that for $r>0, t<T$,  $\p_r \La$ satisfies the equation.
\be
\lab{eqprla}
\begin{cases}
\p^2_r \p_r \La(r, t) - \frac{1}{r} \p_r  \p_r \La(r, t) + g(t) \p_r \p_r \La(r, t) - \p_t \p_r \La(r, t) + \frac{1}{r^2}   \p_r \La(r, t) = 0,\\
\p_r \La(0, t) = 0, \quad \p_r \La(r, 0)=\La'_0(r).
\end{cases}
\ee  Due to the singular zero order term involving $1/r^2$, one needs to handle it with care since the maximum principle does not apply directly.

 First we observe that
\be
\lab{prlainfty}
\liminf_{r \to \infty} \p_r \La(r, t) \ge 0, \quad t \in (0, T).
\ee If not, there exists a $t_0 \in (0, T)$ and a sequence $r_k \to \infty$ such that
\be
\lab{rky0-c}
\liminf_{k \to \infty} \p_r \La(r_k, t_0) = -c <0.
\ee Consider the shifted functions
\[
h_k=h_k(l, t)= \p_r \La(r_k+l, t), \quad l \ge -r_k/2,   \, t \in [0, t_0].
\]Since $g$ is bounded and smooth in $[0, t_0]$ and $\p_r \La(r, 0)=\alpha_0$ for $r \ge 2$, from the standard regularity theory of the heat equation, we know that a subsequence of $\{h_k \}$ converges in $C^{2, 1}_{loc}$ norm, to a smooth solution $h$ of
\[
\begin{cases}
\p^2_l h(l, t) + g(t) \p_l h(l, t) - \p_t h(l, t) =0, \quad (l, t) \in \mathbf{R}^1 \times [0, t_0]\\
h(l, 0)=\alpha_0>0.
\end{cases}
\]Note the terms involving $1/r$ disappears since $r=r_k +l$. But $h(0, t_0)=-c <0$ due to \eqref{rky0-c}, which violate the maximum principle. Hence \eqref{prlainfty} is true. So if $\p_r \La$ changes sign, it must occur in a finite interval of $r$. This property allows us to use a limit argument later on.

For positive integers $i$, we select a smooth, increasing function $\phi_i=\phi_i(r)$:
\[
\phi_i(r)=0,  \, r \in [0, 1/i]; \quad  0 \le \phi_i(r) \le 1, \, r \in (1/i, 2/i); \quad
\phi_i(r)=1, \, r \ge 2/i.
\]We also choose $\phi_i$ in such a way that $\{ \phi_i \}$ is an increasing sequence.

Let $Z_i$ be the bounded and smooth solution to the equation in \eqref{eqprla} in the truncated domain $[1/i, \infty)$:
\be
\lab{eqzi}
\begin{cases}
\p^2_r Z_i(r, t) - \frac{1}{r} \p_r Z_i(r, t)  + g(t) \p_r Z_i(r, t) - \p_t Z_i(r, t) + \frac{1}{r^2}  Z_i(r, t) = 0,  \; r \ge 1/i,   \, t \in [0,T),\\
Z_i(1/i, t) = 0, \quad Z_i(r, 0)=  \phi_i \La'_0(r) \ge 0.
\end{cases}
\ee For each fixed $i$, the coefficients of the equation in the truncated domain are bounded and smooth, therefore the standard maximum principle can be applied to tell us:
\be
\lab{zi>0}
 Z_i(r, t) \ge 0.  \quad r \ge 1/i, \, t \in [0, T).
\ee Here is the detail.  Consider the auxiliary functions $Z_i + \e e^{ A t}$  for a small $\e>0$ and $A >>1$. Then
\be
\lab{eqzi2}
\begin{cases}
\left(\p^2_r  - \frac{1}{r} \p_r   + g(t) \p_r  - \p_t \right) (Z_i + \e e^{ A t})
 =-\frac{1}{r^2}  Z_i(r, t) - \e A e^{ A t},  \; r \ge 1/i,   \, t \in [0,T),\\
Z_i(1/i, t) = 0, \quad Z_i(r, 0)=  \phi_i \La'_0(r) \ge 0.
\end{cases}
\ee  Using the same sliding to infinity method at the beginning of this sub-step, we know that $Z_i(r, t) \ge 0$ when $r$ is sufficiently large. Therefore, if $Z_i + \e e^{ A t}$ becomes negative somewhere, it must first become $0$ at an interior point $(r_0, t_0)$. We set $t_0$ to be the first moment this happens. Then, at this point, we see that
$Z_i(r_0, t_0)= - \e e^{ A t_0}$. The left hand side of \eqref{eqzi2} is non-negative since $(x_0, t_0)$ is a local minimum for $Z_i + \e e^{ A t}$. Hence
\[
0 \le \frac{1}{r^2_0} \e e^{ A t_0}  - A \e  e^{ A t_0}= \left(\frac{1}{r^2_0}   - A \right) \e  e^{ A t_0}.
\] Since $ r_0 \ge 1/i$ and $i$ is fixed, this is a contradiction when $A$ is sufficiently large. Therefore
\[
Z_i + \e e^{ A t} \ge 0,
\]
Letting $\e \to 0$, we deduce $Z_i \ge 0$.

Note that the initial values $\phi_i \La'_0(r)$  are  bounded, nonnegative, smooth and form an increasing sequence. Hence $Z_{i+1}-Z_i$ satisfy the equation
\be
\lab{eqz-zi}
\begin{cases}
\left( \p^2_r  - \frac{1}{r} \p_r   + g(t) \p_r  - \p_t  + \frac{1}{r^2} \right) (Z_{i+1}- Z_i) (r, t) = 0,  \; r \ge 1/i,   \, t \in [0,T),\\
(Z_{i+1}-Z_i)(1/i, t) \ge 0, \quad (Z_{i+1}-Z_i)(r, 0) \ge 0.
\end{cases}
\ee By the maximum principle again, we deduce
\be
\lab{zimono}
Z_{i+1}(r, t) \ge  Z_i(r, t), \quad r \ge 1/i, \, t \in [0, T),
\ee i.e. $\{Z_i \}$ is an increasing sequence.

{\it step 2.2.2.}

Next we show that $Z_i$ are uniformly bounded on compact sets.

Let us define
\be
\lab{delai}
\La_i= \La_i(r, t) = \int^r_{1/i} Z_i(s, t) ds.
\ee Direct computation using \eqref{eqzi} shows:
\[
\p_r \left( \p^2_r \La_i(r, t) - \frac{1}{r} \p_r \La_i(r, t) + g(t) \p_r \La_i(r, t) - \p_t \La_i(r, t) \right) = 0, \qquad (r, t) \in [1/i, \infty) \times [0, T).
\]Hence
\[
\p^2_r \La_i(r, t) - \frac{1}{r} \p_r \La_i(r, t) + g(t) \p_r \La_i(r, t) - \p_t \La_i(r, t) = \p^2_r \La_i(1/i, t),
\]where we have used the identities $\p_r \La_i(1/i, t)= Z_i(1/i, t)=0$ and $\La_i(1/i, t)=0$. Since $\p_r \La_i = Z_i \ge 0$, we also know that the right second derivative
\[
\p^2_r \La_i(1/i, t) \ge 0.
\]Therefore
\[
\begin{cases}
\p^2_r \La_i(r, t) - \frac{1}{r} \p_r \La_i(r, t) + g(t) \p_r \La_i(r, t) - \p_t \La_i(r, t) \ge 0, \qquad (r, t) \in [1/i, \infty) \times [0, T), \\
\La_i(1/i, t)=0, \quad \La_i(r, 0) = \int^r_{1/i} \phi_i(s) \La'_0(s) ds \le \La_0(r).
\end{cases}
\]
This shows that $\La_i$ is a sub-solution of the  equation of  $\La$ in \eqref{eqLa}.
By the maximum principle, which is applicable since the equation has no potential terms, we deduce:
\be
\lab{lai<la}
0 \le \La_i(r, t) \le \La(r, t),   \quad  (r, t) \in [1/i, \infty) \times [0, T).
\ee Thus $\La_i$ are uniformly bounded in compact sets. Now by standard regularity theory (c.f. \cite{Lieb} Ch. IV), we see that
$\{\La_i \}$ is a locally compact sequence in $C^{2, 1, \alpha}_{x, t}$ topology,  in any compact sub-domains of $(0, \infty) \times [0, T)$. Recall that $g=g(t)$ defined in
\eqref{gin1} is smooth in $[0, T)$ so it is H\"older continuous in any compact sub-domain. So the Ascoli-Arzela theorem tells that a subsequence, denoted the same way, converges in the local $C^{2, 1, \alpha}_{x, t}$ norm, to a function $\La_*$, which solves the same equation and the same initial-boundary conditions as $\La$, namely,
\be
\lab{eqLa2}
\begin{cases}
\p^2_r \La_*(r, t) - \frac{1}{r} \p_r \La_*(r, t) + g(t) \p_r \La_*(r, t) - \p_t \La_*(r, t) = 0, \qquad (r, t) \in (0, \infty) \times [0, T),\\
\La_*(0, t)=0, \quad t \ge 0;  \\
\La_*(r, 0)= \La_0(r).
\end{cases}
\ee Observe that due to the second order decay of $\La_0$ near $r=0$, the initial value of $\La_i(r, 0)=\int^r_{1/i} \phi_i(s) \La'_0(s) ds $ actually converge in $C^{1, \alpha}$ norm to $\La_0$.  Let us mention that  Ascoli-Arzela theorem only requires a uniform bound for $\La_i$ in $C^{2, 1, \alpha}_{x, t}$ in every compact domain of $(0, \infty) \times [0, T)$, not in the full domain.

Bear in mind that $0 \le \La(r, t) \le C r^2$ when $r$ is small by step 2.1. This together with \eqref{lai<la} shows that $\La_*$ is a solution which is smooth in the interior and continuous up to the boundary.

 Recall from \eqref{delai} that $Z_i = \p_r \La_i$.  Hence a subsequence of  $\{ Z_i \}$ converges to $\p_r \La_*$ in local $C^{1, \alpha}_{x, t}$ norm, implying
\be
\lab{la2'>0}
\p_r \La_*(r, t) = \lim_{i \to \infty} Z_i(r, t) \ge 0,
\ee since $Z_i \ge 0$ by \eqref{zi>0}.

{\it step 2.2.3.} We prove $\La_*=\La$.

Finally from \eqref{eqLa2} and \eqref{eqLa}, we know that $\La-\La_*$ satisfies the same equation with $0$ initial condition and $0$ boundary value. Moreover $\La-\La_*$ is bounded on compact sets, grows at most linearly as $r \to \infty$, smooth in the interior and continuous up to the boundary. We use the method in Step 1 by considering the function $(\La-\La_*)/r$, which is bounded and satisfies the equation in
\eqref{eqrlam} with $0$ initial value and good signed potential term. Then the maximum principle for that equation again tells us $\La=\La_*$. Since this is a key step, we will present the detail here.
Write
\[
\zeta =\zeta (r, t) = (\La-\La_*)/r.
\]As in \eqref{eqrlam}, $\zeta$ is a solution to the problem in $\mathbf{R}^2 \times [0, T)$:
\be
\lab{eqrlamz}
\Delta \zeta - \frac{1}{r^2} \zeta + g(t) \p_r \zeta + \frac{g(t)}{r} \zeta -\p_t \zeta=0,   \quad \zeta(r, 0)=0.
\ee Here $\Delta=\p^2_r + \frac{1}{r} \p_r$ is the radial Laplacian in $\mathbf{R}^2$ again. Note that $\zeta$ is a bounded, smooth solution. Using the sliding to infinity method in the  sub-step 2.2.1, we see that
$\lim_{r \to \infty} \zeta(r, t)=0$. Hence, if $\zeta$ becomes negative somewhere, it must happen in a compact domain of $\mathbf{R}^2$.   Consider the auxiliary functions $\zeta + \e e^{ A t}$  for a small $\e>0$ and $A >>1$. Then
\be
\lab{eqrlamz2}
\begin{cases}
\left(\Delta    + g(t) \p_r  - \p_t \right) (\zeta + \e e^{ A t})
 =\frac{1}{r^2}  \zeta (r, t)- \frac{g(t)}{r} \zeta - \e A e^{ A t},   \, t \in [0,T),\\
\zeta(r, 0)=  0.
\end{cases}
\ee It suffices to prove that $\zeta(\cdot, t)=0$ for $t \in [0, T-\delta]$ where $\delta$ is any small positive number. So we fix a $\delta>0$ for now.
As explained above, if $\zeta + \e e^{ A t}$ becomes negative somewhere in $\mathbf{R}^2 \times [0, T-\delta]$, it must first become $0$ at an interior point $(r_0, t_0)$. We set $t_0$ to be the first moment this happens. Then, at this point, we see that
$\zeta(r_0, t_0)= - \e e^{ A t_0}$. The left hand side of \eqref{eqrlamz2} is non-negative since $(x_0, t_0)$ is a local minimum for $\zeta + \e e^{ A t}$. Hence
\[
\al
0 &\le -\frac{1}{r^2_0} \e e^{ A t_0} + \frac{g(t_0)}{r_0} \e e^{ A t_0} - A \e  e^{ A t_0}= \left(- \frac{1}{r^2_0}  + \frac{g(t_0)}{r_0} - A \right) \e  e^{ A t_0}\\
&=-\frac{A r^2_0 - g(t_0) r_0 + 1}{r^2_0} \e  e^{ A t_0}=
-\frac{(g(t_0) r_0)^2 - g(t_0) r_0 + 1 + (A-g^2(t_0)) r^2_0}{r^2_0} \e  e^{ A t_0}\\
&=-\frac{(g(t_0) r_0-0.5)^2 + 0.75  + (A-g^2(t_0)) r^2_0}{r^2_0} \e  e^{ A t_0}.
\eal
\] When $A = \sup_{t \in [0, T-\delta]} g^2(t)$ which is finite, the right hand side of the preceding inequality is negative, thus we reach a contradiction. Hence $\zeta \ge - \e e^{A t}$ and  $\zeta \ge 0$ after letting $\e \to 0$. Similarly $\zeta \le 0$. Therefore $\La_*=\La$ for $t \le T-\delta$. Letting $\delta \to 0$, we know  $\La_*=\La$ everywhere.

From,
$
\La_*=\La
$ and  \eqref{la2'>0}, we deduce
\[
\p_r \La \ge 0,
\] completing Step 2.

\medskip

{\it Step 3.} We will show that the function $\La$ is an upper bound for $\Gamma$.

Subtracting $\La$ from the equation for $\Gamma$, we deduce
\be
\lab{eqgala}
\begin{cases}
\al
&(\p^2_r + \p^2_{x_3})  \left( \pm \Gamma - \La \right) - \frac{1}{r} \p_r \left( \pm \Gamma - \La \right) -(v_r, v_3) \cdot (\p_r, \p_{x_3}) \left( \pm \Gamma - \La \right) - \p_t \left( \pm \Gamma - \La \right)\\
 &=(v_r(r, t)+ g(t)) \p_r \La(r, t) , \qquad (r, t) \in (0, \infty) \times [0, T), \quad x_3 \in \mathbf{R}^1\\
&\left( \pm \Gamma - \La \right)(0, t)=0, \quad t \ge 0;  \quad \left( \pm \Gamma - \La \right)(r, 0) \le 0.
\eal
\end{cases}
\ee Since $v_r(r, x_3, t) +g(t) \ge 0$ by our choice  for $g$ in \eqref{gin1} and $\p_r \La \ge 0$, we see that $\pm \Gamma- \La$ is a sub-solution which is non-positive at the parabolic boundary. The maximum principle can be applied to deduce
\be
\lab{g<l}
|\Gamma| \le \La, \quad \qquad (r,  t) \in (0, \infty) \times [0, T), \, x_3 \in \mathbf{R}^1.
\ee
 Since the equation does not have potential terms, we can apply the maximum principle. We caution that some of the terms in the pertinent equations involving $r$ are singular at the boundary $r=0$. However, as long as the solutions are continuous up to the boundary with the correct boundary condition, this singularity do not cause any issue. The reason is that any violation of the maximum principle can only happen in the interior where all functions involved are smooth. So the usual trick, used earlier,  of using axillary functions to establish the maximum principle still works. We will use this property repeatedly.

 Next we look for  a
$C^{2, 1}_{r, t}$ solution $u$ of the following equation, which is obtained from Lemma \ref{lem1ddifs1} using the a priori bound \eqref{gin1}.
\be
\lab{equ2}
\begin{cases}
\p^2_r u(r, t)  + g(t) \p_r u(r, t) - \p_t u(r, t) = 0, \qquad (r, t) \in (0, \infty) \times [0, T),\\
u(0, t)=0, \, t \ge 0;  \quad u(r, 0)= 2 \alpha_0 r, \quad r \ge 0.
\end{cases}
\ee From \eqref{eqLa} and $\p_r \La \ge 0$, we notice that $\La$ is a sub-solution to the above problem. By the maximum principe again, we infer, together with \eqref{g<l}
\be
\lab{glu}
|\Gamma|  \le   \La \le u, \qquad (r,  t) \in (0, \infty) \times [0, T), \, x_3 \in \mathbf{R}^1
\ee Now it is the time to utilize Lemma \ref{lem1ddifs1} on $u$, which infers
\be
\lab{glescr}
|\Gamma(r, x_3, t)|  \le u(r, t) \le C_0 \alpha_0 r^\alpha, \qquad t \ge \delta
\ee for any fixed small $\delta>0$. A rigorous justification of this can be made by working on the time interval $[0, T-\e]$ first. By Lemma \ref{lem1ddifs1} (b), for a fixed $\e>0$, when $r_*=r_*(\e)$ is sufficiently large, we have
\[
|\Gamma(r_*, x_3, t)|  \le u(r_*, t).
\]So the maximum principle  implies \eqref{glescr} on the time interval $[0, T-\e]$.
We remark that the $x_3$ direction does not cause trouble since one can use the sliding method along $x_3$ direction to pull back any possible points where the maximum principle can fail and which are far off in $x_3$ direction. So the maximum principle holds again in the domain $(r, x_3) \in [0, r_*] \times \mathbf{R}^1$.
Since $\e>0$ is arbitrary, we know \eqref{glescr} holds on the full interval $[0, T)$.

Therefore, the swirl velocity $v_\theta$ satisfies the improved bound:
\be
\lab{swjie}
|v_\theta (r, x_3, t)|= |\Gamma(r, x_3, t)|/r \le C_0 \alpha_0 r^{\alpha-1}, \quad \text{if} \quad r \le 1, \quad t \in [\delta, T).
\ee

Now the conclusion of the theorem is an immediate consequence of the main result in \cite{CFZ}, as discussed in the introduction. \qed

\medskip

  Before finishing the paper, we present a separate regularity result for an one dimensional drift diffusion equation where the sign of the drift term is opposite to that in Lemma \ref{lem1ddifs1}. It is not used elsewhere in the paper. The proof is based on a change of variable and De Giogi-Nash-Moser method available in the literature, as in the proof of Lemma \ref{leduk}.   We hope it will be of some use later.
  We mention that a different regularity result for super critical drift diffusion in multi dimensional case has been obtained in \cite{Ib22}, using the method of modulus of continuity and heat kernel estimates.

\begin{proposition}
\lab{pr1ddifs}
Let $u=u(\rho, t)$ be a local $C^{2, 1}_{\rho, t}$ solution to the equation
\be
\lab{eqw}
\begin{cases}
\p^2_\rho u(\rho, t) - g(t) \p_\rho u(\rho, t) - \p_t u(\rho, t) = 0, \qquad (\rho, t) \in [0, \infty) \times [0, T),\\
u(0, t)=0, \quad u(\rho, 0)=u_0(\rho),
\end{cases}
\ee where $g=g(t)$ is in $C([0, T))$ and $g(t) > 0$ for each $t$. Given $\delta \in (0, \sqrt{T})$, there exist absolute  constants $\alpha \in (0, 1)$ and $C_*>0$ which are independent of $g$ and $t$, such that for all $t \in (\delta^2, T)$ and $\rho_0 \ge \delta$, the following spatial H\"older bound holds:
\be
\lab{whold}
\Vert u(\cdot, t) \Vert_{C^\alpha[0, \rho_0]} \le \frac{C_*}{\delta^\alpha}
\Vert u(\cdot, t) \Vert_{L^\infty([0, \rho_0+\delta+\int^t_0 g(s) ds] \times[t-\delta^2, t])}.
\ee
\end{proposition}

\proof
Write
\[
A(t) \equiv \int^t_0 g(s) ds.
\]
 Consider the spatial variable
 \[
 z=\rho - A(t)
 \]and the function
 \[
 \nu=\nu(z, t)=\nu(\rho - A(t), t) \equiv u(\rho, t).
 \]Then the domain for $u$ is
 \[
 P \equiv \{ (z, t) \, | \, z >- A(t), \quad t \in [0, T) \} \subset \mathbf{R}^2_+.
 \]Since
\[
\p^2_\rho u(\rho, t)= \p^2_z \nu(z, t), \quad \p_t u(\rho, t) = \p_t \nu(z, t) - \p_z \nu(z, t) g(t),
\]we find
\[
\al
\p^2_\rho u(\rho, t) - \p_t u(\rho, t) &= \p^2_z \nu(z, t) - \p_t \nu(z, t) + \p_z \nu(z, t) g(t)\\
&=\p^2_z \nu(z, t) - \p_t \nu(z, t) + \p_\rho u(\rho, t) g(t).
\eal
\]This and \eqref{eqw} imply that $\nu=\nu(z, t)$ is a local $C^{2, 1}_{z, t}$ solution to the following one dimensional heat equation with Dirichlet boundary value in the time-varying domain $P$:
\be
\lab{equz}
\begin{cases}
\p^2_z \nu(z, t) - \p_t \nu(z, t) = 0, \qquad (z, t) \in P,\\
\nu(z, t)=0, \,  z= - A(t), \quad t \ge 0; \quad \nu(z, 0)=u_0(z).
\end{cases}
\ee Since $g(t)>0$ by assumption, we know that the left boundary curve $z=-A(t)=-\int^t_0 g(s) ds$ is strictly decreasing and hence it a parabolic boundary of $P$. In addition, the monotone decreasing property of $z=-A(t)$ infers that, for any $r \in (0, \delta]$ and any boundary point $(-A(t), t), t \ge \delta^2$, we have
\be
\lab{q12}
|Q_2(A(t), t, r) \cap P^c | \ge 0.5 |Q_2(A(t), t, r)|
\ee where
\[
Q_2(z, t, r) = \{ (y, s) \, | \, |y-z|<r, \quad t-r^2<s<t \}
\] is the standard parabolic cube is $\mathbf{R}^2$. This property just means that parabolic cubes that straddle the parabolic  boundary of $P$ must intersect the domain $P^c$ by at least a definite portion.

Therefore, the domain $P$ satisfies Condition A, i.e. (6.46) p133 in the book \cite{Lieb}.  Now we can apply Theorem 6.32 in the said book to conclude that for some $\alpha \in (0, 1)$ and $C_*>0$:
\be
\lab{uhold}
\Vert \nu \Vert_{C^{\alpha, \alpha/2}(Q_2(z, t, \delta) \cap P}) \le \frac{C_*}{\delta^\alpha}
\Vert \nu \Vert_{L^\infty(Q_2(z, t, \delta) \cap P)}, \quad (z, t) \in P, \quad t > \delta^2.
\ee
Recall that $\nu(\rho-A(t), t) \equiv u(\rho, t)$. So the H\"older norm in the space direction of the two functions are the same in suitably translated spatial intervals. From \eqref{uhold}, we deduce, for $t \in (\delta^2, T)$ and $\rho \ge 0$, that
\[
\al
\Vert u(\cdot, t) \Vert_{C^\alpha[\rho-\delta, \rho+\delta]} &= \Vert \nu(\cdot, t) \Vert_{C^\alpha[\rho-\delta-A(t), \rho+\delta-A(t)]}\\
& \le \frac{C_*}{\delta^\alpha}
\Vert \nu \Vert_{L^\infty([\rho-\delta-A(t), \rho+\delta-A(t)] \times[t-\delta^2, t])}\\
&=\frac{C_*}{\delta^\alpha}
\Vert u \Vert_{L^\infty( \cup_{s \in [t-\delta^2, t]} [ \rho-\delta - A(t)+A(s), \, \rho+\delta - A(t)+A(s)] \times \{s \})}\\
&\le \frac{C_*}{\delta^\alpha}
\Vert u \Vert_{L^\infty([0, \rho+\delta+A(t)] \times[t-\delta^2, t]) }.
\eal
\]Here $u$ is regarded as $0$ in $[\rho-\delta, 0]$ in case $\rho-\delta<0$.
Covering the spatial interval $[0, \rho_0]$ with finitely many intervals of the form $[\rho-\delta, \rho+\delta]$, we finish the proof of the lemma. Note that the constants are independent of the number of the intervals in the covering. \qed

Finally, we give an example showing that if one replaces $g(t)=c/\sqrt{T-t}$ by $g(t)= \ln(10/(1-t))/\sqrt{1-t}$ in equation \eqref{eqku}. Then the solution $u$ does not have any modulus of continuity at $t=1$. Here we have taken $T=1$ for simplicity. Hence the regularity criteria in the theorem can not be improved by more than a logarithmic term, under the current method.  This example is motivated by \cite{SVZ} and \cite{W}.

Let $\eta: \mathbf{R}^1 \to [0, 1]$ be a smooth function such that $\eta' \le 0$, $\eta(r)=1, \, r \in (-\infty, 0]$, $\eta(r)=0, \, r \in [1, \infty)$ and that $\eta'' \ge - a \eta$, where $a$ is a positive constant. Take
\[
h(t)= \sqrt{1-t} \ln \frac{10}{1-t},
\]
\[
\phi=\phi(r, t)= \exp\left(- a \int^t_0 h^{-2}(s) ds \right) \,  \eta\left((h(t)-r)/h(t) \right).
\]By direct computation, we have
\be
\lab{Lphi}
\al
&(\p^2_r + g(t) \p_r - \p_t) \phi\\
&=(\eta'' + a \eta)  h^{-2}(t) \exp\left(- a \int^t_0 h^{-2}(s) ds \right) \\
&\qquad - \eta'\left((h(t)-r)/h(t) \right) \,  \left[ g(t) + r h'(t) h^{-1}(t) \right] \,
h^{-1}(t) \,
\exp\left(- a \int^t_0 h^{-2}(s) ds \right).
\eal
\ee
 Notice that $\eta' \le 0$ and $\eta'\left((h(t)-r)/h(t) \right)=0$ when $r \ge h(t)$ or $r \le 0$ so that on the support of
$\eta'((h-r)/h)$, one has
 $r/h(t) \in [0, 1]$.
Moreover
\[
h'(t)=-\frac{1}{2 \sqrt{1-t}}  \ln \frac{10}{1-t} + \frac{1}{\sqrt{1-t}} \le 0.
\]
Hence
\[
\al
&- \eta'\left((h(t)-r)/h(t) \right) \,  \left[ g(t) + r h'(t) h^{-1}(t) \right]\\
& \ge - \eta'\left((h(t)-r)/h(t) \right) \,  \left[ g(t) +  h'(t)  \right]\\
&= - \eta'\left((h(t)-r)/h(t) \right) \,  \left[\frac{1}{\sqrt{1-t}} \ln \frac{10}{1-t}  -\frac{1}{2 \sqrt{1-t}}  \ln \frac{10}{1-t} + \frac{1}{\sqrt{1-t}}   \right]\\
& \ge 0.
\eal
\]This inequality and \eqref{Lphi} infer that
\[
(\p^2_r + g(t) \p_r - \p_t) \phi \ge 0
\]Thus $\phi$ is a sub-solution satisfying the initial boundary condition.
\[
\phi(0, t)=0,  \, \phi(r, 0)=\eta(1-(r/\ln 10)).
\]Note that $\int^1_0 h^{-2}(t) dt =1/\ln 10$ which is finite and that
\[
\phi(h(t), t) \ge \exp \left( -a/\ln 10 \right)>0.
\]Since $h(t) \to 0$ as $t \to 1^-$, we conclude that at $t=1$, there is no modulus of continuity for any solution with the same initial and boundary value.

\section*{Acknowledgments} We wish to thank Professors Hongjie Dong, Hussian Ibdah, Zhen Lei, Zijin Li,
  Xinghong Pan, Vladimir Sverak, Xin Yang, Na Zhao and Daoguo Zhou for helpful discussions.
  The support of Simons
  Foundation grant 710364 is gratefully acknowledged.

\noindent {\it Statements and Declarations.} 1. There are no competing interests. 2. There is no data involved in this paper.

\bibliographystyle{plain}



\def\cprime{$'$}

\end{document}